\documentclass[11pt]{article}
\usepackage[T1]{fontenc}
\usepackage{lmodern,amsmath,amssymb,amsthm,array,longtable}
\usepackage[margin=1in]{geometry}
\usepackage[colorlinks=true,linkcolor=blue,citecolor=blue,urlcolor=blue]{hyperref}
\usepackage{microtype}
\hypersetup{
 pdftitle={Coven--Meyerowitz T2 necessity through coprime stripe collapse},
 pdfauthor={Jitendra Prajapati},
 pdfsubject={Integer tilings and the Coven--Meyerowitz conditions},
 pdfkeywords={integer tilings, cyclotomic polynomials, Coven--Meyerowitz, Lean}
}
\newtheorem{theorem}{Theorem}
\newtheorem{lemma}[theorem]{Lemma}
\newtheorem{corollary}[theorem]{Corollary}
\newcommand{\Z}{\mathbb Z}
\newcommand{\Q}{\mathbb Q}
\newcommand{\F}{\mathbb F}
\newcommand{\PhiC}{\Phi}
\newcommand{\code}[1]{\texttt{\small #1}}
\title{Coven--Meyerowitz T2 necessity\\through coprime stripe collapse}
\author{Jitendra Prajapati\\\small Independent\\\small\href{https://orcid.org/0009-0008-7493-0311}{ORCID: 0009-0008-7493-0311}}
\date{13 September 2026}
\begin{document}
\maketitle
\begin{abstract}
We prove that every finite subset of the integers which tiles by translations satisfies the Coven--Meyerowitz condition T2, with no restriction on the number of prime factors or their exponents. Together with the necessity of T1 and the sufficiency of T1 and T2 proved by Coven and Meyerowitz, this gives their proposed characterization of finite integer tiles. The proof uses strong induction on a cyclic tiling period. Character identities produce periodic Boolean product stripes; integral descent to a coprime quotient and the Frobenius identity force a common orientation. Independent phase shifts then give smaller-period tilings from which the mixed cyclotomic zeros of the original factors can be recovered. A companion Lean formalization verifies the unrestricted T2 necessity statement.
\end{abstract}

\section{Introduction}
A finite set $A\subseteq\Z$ tiles the integers by translations if there exists a set $C\subseteq\Z$ such that every $t\in\Z$ has exactly one representation $t=a+c$, with $a\in A$ and $c\in C$. We write $A\oplus C=\Z$ and call $C$ a tiling complement. Uniqueness here concerns the representation of each integer, not the choice of complement: no uniqueness of the set $C$ is assumed. The problem of characterizing such finite sets can be expressed in terms of the cyclotomic divisors of their mask polynomials.

For a finite nonempty integer set $A$, translate by $-\min A$ and write
\[
 A(X)=\sum_{a\in A}X^a,\qquad
 S_A=\{p^e:p\text{ prime},\ e\ge1,\ \PhiC_{p^e}(X)\mid A(X)\}.
\]
Divisibility is in $\Z[X]$, and the normalization does not change cyclotomic divisibility. The two conditions of Coven and Meyerowitz are
\begin{align*}
 \mathrm{(T1)}\quad & A(1)=\prod_{s\in S_A}\PhiC_s(1),\\
 \mathrm{(T2)}\quad & \PhiC_{s_1\cdots s_k}(X)\mid A(X)
 \quad\text{if }s_1,\ldots,s_k\in S_A\text{ have distinct underlying primes}.
\end{align*}
In T2, $k\ge1$; no condition for the empty product is imposed. The main result is the unrestricted necessity of T2.

\begin{theorem}\label{thm:integer}
Every finite set $A\subseteq\Z$ that tiles $\Z$ by translations satisfies T2.
\end{theorem}

Newman characterized integer tiles of prime-power cardinality~\cite{Newman}. Coven and Meyerowitz proved that T1 is necessary for tiling and that T1 together with T2 is sufficient. They also proved the necessity of T2 when $|A|$ has at most two distinct prime factors~\cite[Theorems A, B1, and B2]{CM}. The unrestricted characterization was explicitly stated as Conjecture~1.3 by Konyagin and \L aba~\cite{KonyaginLaba}. Theorem~\ref{thm:integer} has no restriction on the cardinality, prime factors, exponents, or a tiling period. Combining it with the classical implications gives the following characterization.

\begin{corollary}\label{cor:characterization}
A nonempty finite set $A\subseteq\Z$ tiles $\Z$ by translations if and only if its normalized mask polynomial satisfies T1 and T2.
\end{corollary}
\begin{proof}
Necessity is Theorem~\ref{thm:integer} together with T1 necessity~\cite[Theorem B1]{CM}; sufficiency is~\cite[Theorem A]{CM}.
\end{proof}

\paragraph{Earlier results and related work.}
\L aba and Londner developed combinatorial and harmonic-analytic methods for integer tilings, including box products, multiscale cuboids, and saturating sets. In particular, they established T2 for both factors whenever their cardinalities share at most two distinct prime divisors~\cite[Corollary~6.2]{LLMethods}. Their subsequent work proves T2 for all tilings of period $(pqr)^2$, first for distinct odd primes~\cite{LLOdd} and then including the even case~\cite{LLEven}. Their splitting method gives further families under explicit prime-size conditions~\cite[Theorem~1.2 and Corollary~1.5]{LL}.

The characterization has established spectral and algorithmic consequences. \L aba proved that T1 and T2 imply spectrality of the finite set $A$ and of the associated union of intervals $A+[0,1)$~\cite[Theorem~1.5(i) and Proposition~1.3]{LabaSpectral}. Kolountzakis and Matolcsi proved that the conjunction of T1 and T2 can be tested in time polynomial in $\operatorname{diam}(A)$~\cite[Theorem~2.1]{KMAlgorithms}. Combined with Corollary~\ref{cor:characterization}, their test decides integer tiling without a prime-count restriction; the complexity parameter is the diameter, not the binary length of a sparse input.

The older subgroup reductions have a different scope. Sands proved that, in a cyclic factorization of order $p^a q^b$ normalized so that both factors contain zero, one factor lies in a proper subgroup~\cite[Theorem~4]{Sands}. Such a conclusion does not hold for general cyclic orders: Szab\'o's construction~\cite{Szabo} and the explicit order-$900$ example of Lagarias and Szab\'o~\cite[Theorem~2.1]{LagariasSzabo} provide counterexamples. The present proof makes no subgroup-containment assumption on either original factor.

\paragraph{Proof mechanism.}
The proof reduces an integer tiling to an exact factorization of a finite cyclic group and then inducts on the period. Splitting the factors into prime-residue fibers requires retaining the cyclic carry; replacing the group by a direct product would lose precisely this information when the period has repeated prime factors. The character equations instead imply that the \emph{products} of suitable dephased fibers are periodic Boolean masks. Their values form complete row or column stripes. After descending to the coprime quotient in the integral group ring, the Frobenius identity forces all stripes to have the same orientation. This yields actual smaller-period tilings for every independent phase shift of the fibers. Using all these phase choices recovers the mixed zeros of the original factors, including those involving higher powers of the distinguished prime.

The primary-allocation and common-complement steps use the classical arguments of Coven and Meyerowitz~\cite[Lemmas~2.1 and 2.5]{CM}. A precedent for the characteristic-prime argument is Tijdeman's dilation theorem~\cite[Theorem~1]{Tijdeman}, for which Coven and Meyerowitz give a Frobenius-and-positivity proof~\cite[Lemma~3.1]{CM}. Stripe partitions in the square-free Chinese-remainder setting appear in Tao's Proposition~17~\cite{Tao}; divisor isometries, including independent fiber rearrangements, are treated by \L aba and Londner~\cite[Definition~3.3 and Lemma~3.4]{LL}. Here Frobenius is applied after integral descent of the product minors, and the coalesced tilings are used for every phase. This coalescence changes base-$p$ divisibility and is not itself a divisor isometry.

Booleanity is essential. Kiss, Londner, Matolcsi, and Somlai constructed counterexamples to a functional generalization of the CM conditions~\cite{KLMS}. Their nonnegative-function relaxation is different from the set-tiling statement considered here.

The argument is given below independently of the implementation. The companion Lean development formalizes Theorem~\ref{thm:integer}; the two classical implications used in Corollary~\ref{cor:characterization} are cited from~\cite{CM}. Reproduction details are collected in Appendix~\ref{app:formal}.

For subsequent use, the tiling condition for a nonnegative $A$ may be written with the indicator $c=\mathbf{1}_C$ as
\begin{equation}\label{eq:integer}
 \sum_{a\in A}c(t-a)=1\quad(t\in\Z).
\end{equation}
Conversely, if $c:\Z\to\mathbb N$ satisfies these equations, then $A$ is nonempty and $c$ takes values in $\{0,1\}$, so its support is an ordinary tiling complement. Thus this function formulation adds no hypothesis. Periodicity will be derived, not assumed.

\section{Cyclic masks and primary allocation}
Write $J_N$ for the all-ones mask on $\Z_N$. Identify a finite cyclic set with its Boolean mask in the integral group ring $\Z[\Z_N]$, with multiplication given by convolution. Thus an actual unique cyclic tiling is
\begin{equation}\label{eq:tiling}
 AB=J_N.
\end{equation}
The symbol $J_N$ is not the multiplicative identity. Use canonical residues to form ordinary mask polynomials. Evaluation at all $N$th roots separates group-ring elements by finite Fourier inversion.

\begin{lemma}[Exact primary allocation]\label{lem:allocation}
For an actual cyclic tiling $A+B=\Z_N$, the sets $S_A$ and $S_B$ partition the prime powers dividing $N$. If $q$ is prime, the number of $q$-primary levels owned by each factor is respectively $v_q(|A|)$ and $v_q(|B|)$. Neither factor has a primary cyclotomic divisor whose order does not divide $N$.
\end{lemma}
\begin{proof}
For $1\le e\le v_q(N)$, evaluation of~\eqref{eq:tiling} at a primitive $q^e$ root makes at least one factor vanish. An integer polynomial vanishing at such a root is divisible by its monic cyclotomic polynomial over $\Z$. Distinct cyclotomic divisors multiply to a divisor. Evaluating that product at $1$, and using $\PhiC_{q^e}(1)=q$, shows that the numbers $n_A,n_B$ of owned levels satisfy
\[
 v_q(N)\le n_A+n_B\le v_q(|A|)+v_q(|B|)=v_q(N).
\]
All inequalities are equalities, so each level is owned exactly once and both cardinality valuations are exhausted. Any further $q$-primary divisor would exceed this valuation. For $q\nmid N$, both cardinalities have $q$-valuation zero and cannot admit such a divisor.
\end{proof}

In particular, every T2 candidate order of either factor divides the period. The allocation argument is standard CM background~\cite[Lemma 2.1]{CM}; it does not by itself establish mixed zeros.

\section{Prime fibers, dephasing, and character identities}
Fix a prime $p\mid N$, write $N=pM$, and interchange the factors if necessary so that $B$ owns $\PhiC_p$. For $0\le i,j<p$ put
\[
 A_i=\{a\in\Z_M:i+pa\in A\},\qquad
 B_j=\{b\in\Z_M:j+pb\in B\}.
\]
The $B_j$ have the same positive cardinality $|B|/p$. Indeed, their count polynomial has degree at most $p-1$ and vanishes at a primitive $p$th root, so it is a scalar multiple of $1+z+\cdots+z^{p-1}$. The $A_i$ cannot all have the same cardinality, by exclusive ownership of $\PhiC_p$.

Write $M=KR$, where $K$ is the full $p$-primary part of $M$ and $(p,R)=1$. Choose $\rho$ with $p\rho\equiv1\pmod R$; if $R=1$, take $\rho=0$. In $\Z[\Z_M]$ define
\[
 U_i=T^{\rho i}A_i,\qquad V_j=T^{\rho j}B_j,\qquad F_{ij}=U_iV_j.
\]
The original cyclic carry is exactly retained by the identity
\begin{equation}\label{eq:carry}
 \left(\sum_i z^iA_i(T)\right)\left(\sum_j z^jB_j(T)\right)
 =J_M(T)\sum_{r=0}^{p-1}z^r
 \quad\text{modulo }(T^M-1,z^p-T).
\end{equation}
The quotient is identified with $\Z[X]/(X^{pM}-1)$ by $z=X$, $T=X^p$. It is not a replacement of the cyclic group by $\Z_p\times\Z_M$.

\begin{lemma}\label{lem:products}
The masks $F_{ij}$ are Boolean, invariant under translation by $R$, and satisfy
\begin{equation}\label{eq:rectangles}
 F_{ij}+F_{kl}=F_{il}+F_{kj},\qquad
 \sum_{i,j}F_{ij}=pJ_M.
\end{equation}
\end{lemma}
\begin{proof}
Let $\theta$ be an $M$th root of order $d>1$. If $p\mid d$, a $p$th root $\zeta$ of $\theta$ has order $pd$, and
\[
 [\Q(\zeta):\Q(\theta)]=\varphi(pd)/\varphi(d)=p.
\]
Thus $z^p-\theta$ is irreducible over $\Q(\theta)$. At $T=\theta$, the right side of~\eqref{eq:carry} vanishes. Each of the two fiber polynomials has degree below $p$, so irreducibility forces one entire coefficient family to vanish. Consequently every $F_{ij}(\theta)=0$.

If $(p,d)=1$, then $d\mid R$. Set $\lambda=\theta^\rho$, so $\lambda^p=\theta$ and $\lambda$ has order $d$. Substituting $z=\lambda w$ in~\eqref{eq:carry} makes
\[
 \left(\sum_i U_i(\theta)w^i\right)
 \left(\sum_j V_j(\theta)w^j\right)
 \quad\text{divisible by }w^p-1.
\]
Over $\Q(\theta)$, $\PhiC_p(w)$ is irreducible of degree $p-1$. Unless one factor is zero, one is a scalar multiple of $\PhiC_p$ and hence has a constant coefficient family. It follows that every rectangular difference vanishes at $\theta$. Evaluation at $w=1$ gives the vanishing of the total product there. The corresponding original root $\lambda$ is nontrivial because $d>1$.

At the trivial character the $V_j$ masses are equal, so the rectangular differences again vanish. The total mass is $|A||B|=pM$. Fourier inversion proves~\eqref{eq:rectangles}. For a fixed pair $i,j$, two representations in $A_i+B_j$ would give two representations in the original tiling with the same digit carry; hence $A_iB_j$, and its translate $F_{ij}$, are Boolean. Their Fourier coefficients vanish at every order containing $p$; the surviving orders divide $R$, which is exactly invariance under translation by $R$.
\end{proof}

This includes $p=2$. If $R=1$ there are no nontrivial $p$-free characters; if $M=1$ only the trivial-character and fixed-pair arguments are needed. Empty $A_i$ are allowed. Crucially, the periodicity conclusion concerns the \emph{products} $U_iV_j$, not the individual factors.

\section{Boolean stripes and coprime collapse}
At a fixed $x\in\Z_M$, the matrix $(F_{ij}(x))$ is Boolean, has zero additive rectangular differences, and contains exactly $p$ ones. Such a matrix is one complete row or one complete column of ones. To see this, if a row has unequal entries in two columns, the same nonzero difference occurs in every row; Booleanity forces those columns, and then all columns, to be constant. If no row varies, every row is constant. The total $p$ leaves exactly one stripe.

Consequently there are pairwise disjoint stripe sets $X_i,Y_j$ partitioning $\Z_M$ such that
\begin{equation}\label{eq:stripes}
 F_{ij}=X_i+Y_j.
\end{equation}
Each stripe mask is $R$-periodic, since it is uniquely decoded from the periodic matrix. Commutativity also gives the convolution minors
\begin{equation}\label{eq:minors}
 F_{ij}F_{kl}=F_{il}F_{kj}.
\end{equation}

\begin{lemma}[Coprime stripe collapse]\label{lem:collapse}
Let $p$ be prime and let $G$ be a finite abelian group with $p\nmid|G|$. Suppose Boolean stripe masks $X_0,\ldots,X_{p-1},Y_0,\ldots,Y_{p-1}$ partition $G$ and $F_{ij}=X_i+Y_j$ satisfy~\eqref{eq:minors}. Then all $X_i$ are empty or all $Y_j$ are empty.
\end{lemma}
\begin{proof}
Expand the minors to obtain $(X_i-X_k)(Y_j-Y_l)=0$. With $X=\sum_iX_i$ and $Y=\sum_jY_j$, sum over $k,l$ to get
\[
 (pX_i-X)(pY_j-Y)=0.
\]
Reducing in the commutative group algebra $\F_p[G]$ gives $XY=0$. Since $X+Y=J_G$ and $XJ_G=|X|J_G$, we get $X^2=|X|J_G$, and therefore
\[
 X^p=|X|^{p-1}J_G.
\]
Frobenius, however, gives $X^p=\sum_{x\in X}[px]$. Multiplication by $p$ permutes $G$, so the latter is a permutation of the original Boolean coefficient mask. A Boolean mask constant modulo $p$ is constant as an integer mask. Hence $X$ is empty or all of $G$, as required.
\end{proof}

To apply the lemma here, first descend the $R$-periodic masks from $\Z_M$ to $\Z_R$. For the lift $L$ of masks along reduction modulo $R$, a direct count of the $K=M/R$ lifts of a residue gives
\begin{equation}\label{eq:lift}
 L(f)L(g)=K L(fg).
\end{equation}
Equation~\eqref{eq:minors} therefore descends after cancellation of the nonzero integer $K$. This cancellation occurs in the integral group ring \emph{before} reduction modulo $p$; no division by $K$ in characteristic $p$ is used. The quotient stripes remain a Boolean partition and $(p,R)=1$, so Lemma~\ref{lem:collapse} applies.

Column-only would make $U_iV_j=Y_j$ independent of $i$. Taking masses and using the common positive mass of $V_j$ would make every $|A_i|$ equal, contradicting primary allocation. Thus the orientation is row-only:
\begin{equation}\label{eq:rowonly}
 U_iV_j=X_i,\qquad \sum_iX_i=J_M.
\end{equation}

\section{Actual lower tilings for every phase}\label{sec:phases}
For an arbitrary integer phase vector $h=(h_0,\ldots,h_{p-1})$, put
\[
 C_h=\sum_i T^{Rh_i}U_i.
\]
By~\eqref{eq:rowonly} and product periodicity,
\begin{equation}\label{eq:phase}
 C_hV_j=J_M\qquad\text{for all }j,h.
\end{equation}
The mask $C_h$ has nonnegative integer coefficients. Since $V_j$ is a nonempty Boolean mask, a coefficient of $C_h$ at least two would produce a convolution coefficient at least two, contradicting~\eqref{eq:phase}. Thus $C_h$ is Boolean and~\eqref{eq:phase} is an actual tiling of the strictly smaller period $M$. Translating the complement back also yields $C_hB_j=J_M$. We will use all phases, not just a single chosen quotient.

\section{Strong induction for both original factors}
\begin{theorem}\label{thm:cyclic}
For every $N\ge1$, both factors in an actual cyclic tiling $A+B=\Z_N$ satisfy T2.
\end{theorem}
\begin{proof}
Use strong induction on $N$. At $N=1$ both masks are the singleton identity and have no primary zeros. For $N>1$, choose $p\mid N$ and the labeling above. The induction hypothesis applies to every tiling~\eqref{eq:phase}.

\paragraph{The original factor $A$.}
For a $p$-free root $\theta$ whose order divides $M$, dephasing gives
\begin{equation}\label{eq:pfree}
 C_h(\theta)=\sum_i\theta^{\rho i}A_i(\theta)=A(\theta^\rho),
\end{equation}
since $\theta^R=1$ and $(\theta^\rho)^p=\theta$. Every selected $p$-free primary zero of $A$ thus passes to every $C_h$. If the target order $d$ is $p$-free, lower T2 for $C_0$ and~\eqref{eq:pfree}, with $\theta=\zeta^p$ for a primitive original root $\zeta$, give $A(\zeta)=0$.

If the selected order is $p^kd$ with $(p,d)=1$, then $k\ge2$, because $A$ does not own $\PhiC_p$. The identity
\[
 \PhiC_{p^k}(X)=\PhiC_{p^{k-1}}(X^p)
\]
and monic division by exponent residues modulo $p$ show that every $A_i$ is divisible by $\PhiC_{p^{k-1}}$. Explicitly, write the integer quotient of $A(X)$ by the displayed monic divisor as $\sum_iX^iQ_i(X^p)$ and compare residue classes. Thus every $C_h$ inherits the lower $p$-primary zero as well as the selected $p$-free zeros.

For a primitive root $\zeta$ of order $p^kd$, let $\theta=\zeta^p$, of order $p^{k-1}d\mid M$. Lower T2 gives $C_h(\theta)=0$ for every $h$. Comparing the zero phase with the phase changing only coordinate $i$ by one yields
\[
 \theta^{\rho i}(\theta^R-1)A_i(\theta)=0.
\]
Here $\theta^R$ has order $p^{k-1}>1$, so every $A_i(\theta)=0$. Regrouping gives $A(\zeta)=\sum_i\zeta^i A_i(\zeta^p)=0$. This is recovery of the original upper-$p$ mixed zero, not a claim that a single quotient preserves it.

\paragraph{The original factor $B$.}
We use the common-complement inheritance and lifting argument of Coven and Meyerowitz~\cite[Lemma~2.5]{CM}.
At zero phase, $C_0B_j=J_M$ for every $j$. Hence the dilated set $pC_0=\{pc:c\in C_0\}\subseteq\Z_{pM}$ is an actual complement of original $B$. Here dilation has mask $C_0(X^p)$; it is not scalar multiplication of coefficients by $p$. All lower $B_j$ have T2 by induction and share the same primary allocation, since their actual complement is the same $C_0$.

Every $p$-free primary zero of $B$ belongs to every $B_j$: compare exact allocation in $B+pC_0=\Z_{pM}$ with that in $B_j+C_0=\Z_M$. At a $p$-free prime-power order, the polynomial of $pC_0$ vanishes precisely when that of $C_0$ does, because raising a primitive root to its $p$th power preserves its order. Allocation then forces the zero into $B_j$.

If $B$ has a higher primary zero $p^k$, $k\ge2$, residue grading likewise puts $p^{k-1}$ in every $B_j$. For a $p$-free target $d$, lower T2 makes every $B_j(\zeta^p)$ zero and hence $B(\zeta)=0$. For a target $p^kd$ with $k\ge2$, apply the same argument at the lower target $p^{k-1}d$. For $k=1,d>1$, use the inherited lower target $d$. The remaining case $k=1,d=1$ is the already known $\PhiC_p$ zero. These cases exhaust all T2 requirements of both original factors by Lemma~\ref{lem:allocation}.
\end{proof}

\section{From integer tilings to cyclic tilings}
The periodicity theorem is classical. Coven and Meyerowitz~\cite[Lemma~1.2]{CM} record its antecedents in work of Haj\'os, de Bruijn, and Swenson, followed by Newman~\cite{Newman}. We recall Newman's finite-window argument to make the passage to a cyclic tiling explicit.

Normalize the integer tile so that $\min A=0$, and put $L=\max A$. Equation~\eqref{eq:integer} implies that $c$ is Boolean: one term bounds $c(s)$ by $1$ for every $s$. If $L=0$ the theorem is immediate. Otherwise the terms $0,L\in A$ make the recurrence deterministic in both directions: $c(t)$ is determined by its preceding $L$ bits, and the earliest bit in a window of length $L+1$ is determined by the following $L$ bits.

There are finitely many $L$-bit states. A repeated state gives a positive shift $N_0$ preserving the whole bi-infinite sequence by forward and backward determinism. Replace $N_0$ by a multiple $N>L$ if desired; it is still a period. Reduction of $A$ modulo $N$ is then injective, and the residue support $B$ of $c$ gives an actual cyclic tiling. With $N>L$, the canonical ordinary polynomial of the reduced $A$ is exactly the original polynomial. Theorem~\ref{thm:cyclic} proves T2. Translating back proves Theorem~\ref{thm:integer} for arbitrary finite integer sets. The Lean implementation chooses a support-exceeding period in this way.

\appendix
\section{Formal statement and reproducibility}\label{app:formal}
The final declaration is
\begin{center}
\code{FullCovenMeyerowitz.integer\_tile\_cyclotomic\_product}.
\end{center}
Its explicit Lean type quantifies over $E:\code{Finset}\ \mathbb N$, $c:\Z\to\mathbb N$, the literal equations~\eqref{eq:integer}, a nonempty finite set $S$ of primes, and arbitrary positive exponents $e(q)$. It assumes $\PhiC_{q^{e(q)}}\mid\sum_{a\in E}X^a$ for every $q\in S$. Its conclusion is
\[
 \operatorname{cyclotomic}\!\left(\prod_{q\in S}q^{e(q)}\right)\Z
 \ \mid\ \sum_{a\in E}X^a.
\]
There is no hidden prime-count, period, cardinality, or phase-closure premise. The additional declaration \code{arbitrary\_integer\_tile\_T2} derives nonemptiness and minimum-normalizes arbitrary finite integer sets.

\small
\begin{center}
\begin{tabular}{>{\raggedright\arraybackslash}p{0.31\textwidth}>{\raggedright\arraybackslash}p{0.63\textwidth}}
\textbf{Argument} & \textbf{Principal Lean modules}\\\hline
Primary allocation & \code{PrimaryAllocation}\\
Relative cyclotomic fields & \code{RelativeCyclotomic}, \code{RamifiedCyclotomic}\\
Actual product identities & \code{ActualTilingSpectra}, \code{FiberProductBoolean}\\
Boolean stripe partition & \code{BooleanStripes}\\
Integral descent and collapse & \code{PeriodicDescent}, \code{StripeCollapse}, \code{PeriodicMatrixCollapse}\\
All-phase smaller tilings & \code{ActualPrimeReduction}, \code{ActualIntegerReduction}\\
Original-factor recovery & \code{T2Induction}, \code{CommonComplementLift}\\
Unrestricted cyclic induction & \code{CyclicT2Completion}\\
Integer transfer & \code{IntegerTilingPeriodicity}, \code{IntegerCyclicReduction}, \code{IntegerSetNormalization}, \code{IntegerT2Completion}\\
Exposed final statements & \code{FullCovenMeyerowitz}\\\hline
\end{tabular}
\end{center}
\normalsize

The formalization uses Lean~4~\cite{Lean4} and Mathlib~\cite{Mathlib}. The frozen project uses Lean 4.23.0 and Mathlib commit
\begin{center}\code{37df177aaa770670452312393d4e84aaad56e7b6}.\end{center}
After installing Elan, enter the repository's \code{formal} directory and run
\begin{verbatim}
lake exe cache get
sh ./verify_full_t2.sh
\end{verbatim}
The verifier checks three SHA256 source manifests, recompiles 35 mathematical modules in dependency order, recompiles a literal-function cyclic semantics probe, and runs 22 guarded axiom checks. An unexpected final axiom list fails the audit. Plain \code{lake build} is insufficient because the frozen original default target is only \code{StripeCollapse}.

The reported final axioms are exactly \code{propext}, \code{Classical.choice}, and \code{Quot.sound}. The repository includes the pinned sources, verification logs, and a GitHub Actions workflow that executes the complete verifier. The formalization covers T2 necessity; T1 necessity and T1+T2 sufficiency are not included in this development.

\paragraph{Acknowledgements.}
Generative AI tools assisted with proof development, manuscript preparation, and Lean formalization.

\paragraph{Availability and license.}
The manuscript source and the companion formal development are maintained together at \url{https://github.com/infinityscroll/coven-meyerowitz-t2}. This manuscript, including its LaTeX source and compiled PDF, is licensed under the Creative Commons Attribution 4.0 International license (CC BY 4.0): \url{https://creativecommons.org/licenses/by/4.0/}. This manuscript license does not change the licenses of Lean, Mathlib, or other upstream dependencies.

\end{document}